\documentclass{article}
\usepackage{graphicx} 
\usepackage[english]{babel}

\usepackage[letterpaper,top=2cm,bottom=2cm,left=2cm,right=2cm,marginparwidth=1.75cm]{geometry}
\usepackage{csquotes}
\usepackage[colorlinks=true, allcolors=blue]{hyperref}

\usepackage{chngcntr}
\usepackage{makecell}
\usepackage{graphicx}
\usepackage{longtable}
\usepackage[nottoc]{tocbibind}
\usepackage{enumitem}
\usepackage{lscape}
\graphicspath{ {C:/Users/steve/Documents/4H project} }
\usepackage{amsmath}
\usepackage{amssymb}
\usepackage{bbm}
\usepackage{float}
\usepackage{amsfonts}
\usepackage{amsthm}
\usepackage{listings}
\usepackage{xcolor}
\usepackage{mwe}
\definecolor{codegreen}{rgb}{0,0.6,0}
\definecolor{codegray}{rgb}{0.5,0.5,0.5}
\definecolor{codepurple}{rgb}{0.58,0,0.82}
\definecolor{backcolour}{rgb}{0.95,0.95,0.92}

\lstdefinestyle{mystyle}{
    backgroundcolor=\color{backcolour},   
    commentstyle=\color{codegreen},
    keywordstyle=\color{magenta},
    numberstyle=\tiny\color{codegray},
    stringstyle=\color{codepurple},
    basicstyle=\ttfamily\footnotesize,
    breakatwhitespace=false,         
    breaklines=true,                 
    captionpos=b,                    
    keepspaces=true,                 
    numbers=left,                    
    numbersep=5pt,                  
    showspaces=false,                
    showstringspaces=false,
    showtabs=false,                  
    tabsize=2
}

\usepackage{mathtools}
\usepackage{lscape}
\usepackage{amssymb}
\usepackage{caption}
\usepackage{array}
\usepackage[format=plain,
            labelfont=it,
            textfont=it]{caption}
\usepackage{tikz-cd}
\usepackage{mathdots}
\usetikzlibrary{decorations.markings}
\tikzset{negated/.style={
        decoration={markings,
            mark= at position 0.5 with {
                \node[transform shape] (tempnode) {$\backslash$};
            }
        },
        postaction={decorate}
    }
}
\usepackage{bm}

\newcolumntype{M}[1]{>{\centering\arraybackslash}m{#1}}
\newcolumntype{N}{@{}m{0pt}@{}}

\newtheorem{theorem}{Theorem}[section]
\newtheorem*{theorem*}{Theorem A}
\newtheorem{cor}[theorem]{Corollary}
\newtheorem{lemma}[theorem]{Lemma}

\newtheorem{example}[theorem]{Example}
\theoremstyle{definition}

\newtheorem{definition}[theorem]{Definition}
\newtheorem{conjecture}[theorem]{Conjecture}

\newcommand{\nth}{\textup{th}}
\newcommand{\N}{\mathbb{N}}

\newcommand{\R}{\mathbb{R}}

\newcommand{\F}{\mathbb{F}}

\newcommand{\Z}{\mathbb{Z}}
\newcommand{\K}{\mathbb{K}}

\newcommand{\intdist}[1]{|\langle#1\rangle|}

\numberwithin{equation}{section}
\title{Counterexamples to the $p(t)$-adic Littlewood Conjecture Over Small Finite Fields}
\date{}
\author{Samuel Garrett, Steven Robertson}
\begin{document}
\maketitle
\begin{abstract}
    \noindent In 2004, de Mathan and Teulié stated the $p$-\textit{adic Littlewood Conjecture} ($p$-$LC$) in analogy with the classical Littlewood Conjecture. Given a field $\K$ and an irreducible polynomial $p(t)$ with coefficients in $\K$, $p$-$LC$ admits a natural analogue over function fields, abbreviated to $p(t)$-$LC$ (and to $t$-$LC$ when $p(t)=t$).\\
    
    \noindent In this paper, an explicit counterexample to $p(t)$-$LC$ is found over fields of characteristic 5. This is one example of an infinite family of Laurent series that are conjectured to disprove $p(t)$-$LC$ over all fields of odd characteristic. This fills a gap left by a breakthrough paper from Adiceam, Nesharim and Lunnon (2022) in which they conjecture $t$-$LC$ does not hold over all complementary fields of characteristic $p\equiv 3\mod 4$ and prove this in the case $p=3$. Supported by computational evidence, this provides a complete picture on how $p(t)$-$LC$ is expected to behave over all fields with characteristic not equal to 2. Furthermore, the counterexample to $t$-$LC$ over fields of characteristic 3 found by Adiceam, Nesharim and Lunnon is proven to also hold over fields of characteristic 7 and 11, which provides further evidence to the aforementioned conjecture.\\
    
    \noindent Following previous work in this area, these results are achieved by building upon combinatorial arguments and are computer assisted. A new feature of the present work is the development of an efficient algorithm (implemented in Python) that combines the theory of automatic sequences with Diophantine approximation over function fields. This algorithm is expected to be useful for further research around Littlewood-type conjectures over function fields. 
\end{abstract}
\section{Introduction}
Let $\alpha$ be a real number and let $|\alpha|$ denote its usual absolute value. Define the \textbf{distance to the nearest integer} of $\alpha$ as $|\langle\alpha\rangle|=\min_{n\in\Z}\{|\alpha-n|\}$. The famous \textbf{Littlewood Conjecture} (1930's) states that for any pair of real numbers $(\alpha,\beta)$, \begin{equation}\label{LC}
    \liminf_{n\to\infty}|n|\cdot |\langle n\alpha\rangle|\cdot\intdist{n\beta}=0
\end{equation}
\noindent From now on, ``Littlewood's Conjecture'' is abbreviated to $LC$. In 2006, Einsiedler, Katok and Lindenstrauss \cite{EKL} proved that the set of counterexamples to (\ref{LC}) has Hausdorff dimension equal to zero. This remains the most significant progress to date towards proving $LC$. In 2004, de Mathan and Teulié \cite{padic} formulated a similar conjecture that has become known as the $p$-\textbf{adic Littlewood Conjecture} ($p$-$LC$). Recall that, for a given prime $p$ and natural number $n$, the $p$-\textbf{adic norm} of $n$ is defined as \begin{equation}\label{p-adic}
    |n|_p=p^{-\nu_p(n)}, ~~\text{ where }~~\nu_p(n):=\max_{i\in\N}\{i: p^i|n\}.
\end{equation}
\begin{conjecture}[$p$-$LC$, de Mathan and Teulié, 2004]
    For any real number $\alpha$ and any prime $p$, \begin{equation}
        \inf_{n\in\N\backslash\{0\}}|n|\cdot |n|_p\cdot \intdist{n\alpha} = 0.
    \end{equation}
\end{conjecture}
\noindent Similarly to $LC$, Einsiedler and Kleinbock \cite{2007} proved in 2007 that the set of counterexamples to $p$-$LC$ has Hausdorff dimension equal to zero. For a deeper exploration of $LC$ and $p$-$LC$, see \cite{complexity,around_Littlewood,mult_dio} and  \cite{Intro_Littlewood}. Both of the aforementioned conjectures admit natural counterparts over function fields sitting over a ground field $\K$, with the focus of this paper being on the analogue of $p$-$LC$. Since the counterparts of $LC$ and $p$-$LC$ are known to be false when $\K$ is infinite (see \cite{t-adic_tm,CE} and \cite{padic}), it is assumed throughout this paper that the ground field $\K$ is finite.\\

\noindent Let then $\mathbb{K}=\F_q$ be the finite field with cardinality $q$. Denote the ring of polynomials with coefficients in $\F_q$ as $\F_q[t]$ and let $\F_q(t)$ be the field of rational functions. The \textbf{norm}\footnote{Although the norm over $\R$ and $\F_q((t^{-1}))$ share the same notation, it will be clear from context which is being used.} of $\Theta(t)\in\F_q(t)$ is defined as \begin{equation}\label{norm}
    |\Theta(t)|=q^{\deg(\Theta(t))}.
\end{equation}
This metric is used to define the completion of $\F_q(t)$, namely the field of formal Laurent series \begin{equation*}
    \F_q((t^{-1}))=\left\{\sum_{i=-h}^\infty a_i t^{-i}: h\in\Z\text{ and }a_i\in\F_q\right\}.
\end{equation*} 
\noindent In the function field set up, the analogue of a prime number is an irreducible polynomial $p(t)\in\F_q[t]$. Similarly to (\ref{p-adic}), define the $p(t)$\textbf{-adic norm} of a polynomial $N(t)\in\F_q[t]$ as \begin{equation*}
    |N(t)|_{p(t)}=|p(t)|^{-\nu_{p(t)}(N(t))}, ~~\text{  where  }~~\nu_{p(t)}(N(t))=\max_{i\in\N}\{i: p(t)^i|N(t)\}.
\end{equation*}
The analogue of $p$-$LC$, called the $p(t)$\textbf{-adic Littlewood Conjecture} ($p(t)$-$LC$), is then: \begin{conjecture}[$p(t)$-$LC$]\label{p(t)-LC}
    Let $\F_q$ be a finite field. For every $\Theta(t)$ in $\F_q((t^{-1}))$ and every irreducible $p(t)$ in $\F_q[t]$, \begin{equation*}
        \inf_{\N(t)\in\F_q[t]\backslash\{0\}} |N(t)|\cdot |N(t)|_{p(t)}\cdot \intdist{N(t)\cdot \Theta(t)}=0.
    \end{equation*}
\end{conjecture}

\noindent When $p(t)=t$, $p(t)$-$LC$ is further abbreviated to $t$-$LC$.\\

\noindent Not all results about $p$-$LC$ translate into the positive characteristic set up. In 2020, Einsiedler, Lindenstrauss
and Mohammadi \cite{PC} provided the function field analogues of the measure classification results of both Lindenstrauss, Einsiedler and Katok \cite{EKL} and Einsiedler and Kleinbock \cite{2007}. However, it is not clear if these results are sufficient for proving anything about the Hausdorff dimension of the set of counterexamples to $p(t)$-$LC$. \\

\noindent This paper focusses on $p(t)$-LC over finite fields of odd characteristic. Nothing conclusive is known when the ground field has characteristic 2, but Section \ref{sect:end} provides some evidence that $t$-LC may be true over $\F_2$.\\

\noindent The most significant progress towards testing the validity of Conjecture \ref{p(t)-LC} is due to a breakthrough paper by Adiceam, Nesharim and Lunnon \cite{Faustin} that provided an explicit counterexample to $t$-$LC$ over fields of characteristic 3. The sequence of coefficients of this Laurent series, denoted $\Upsilon(t)$, is the Paperfolding sequence. Whilst the Paperfolding sequence has been studied extensively, the larger family of $n^\nth$-Level Paperfolding sequences introduced hereafter (see Definition \ref{npf}) is original to this paper. The classic Paperfolding Sequence corresponding to the $n=1$ case of this definition.
\begin{definition}\label{npf} Let $n$ be a natural number, $\nu_p(n)$ be as in (\ref{p-adic}) and let $\overline{~\cdot~}:\N\to\N$ be the value of the input modulo $2^{n+1}$. Define the $n^\nth$-\textbf{Level Paperfolding Sequence} $(\Pi_i^{(n)})_{i\ge1}$ as\begin{equation}\label{PFseq}
    \Pi^{(n)}_{i}= \frac{\overline{2^{-\nu_2(i)}i}-1}{2}\cdotp
\end{equation}\end{definition} 

\noindent Computer evidence provided by Adiceam, Nesharim and Lunnon \cite{Faustin} provides strong support for the following conjecture:\begin{conjecture}[Adiceam, Nesharim, Lunnon]\label{p=3}
    Let $p\equiv 3\mod 4$ be a prime. Then, $\Upsilon(t)$ is a counterexample to $t$-$LC$ over fields of characteristic $p$.
\end{conjecture}  \noindent The same computations nevertheless show that it is extremely unlikely that $\Upsilon(t)$ should provide a counterexample to $t$-$LC$ in the complementary case when the ground field has characteristic $p\equiv 1 \mod 4$. The following theorem fills this gap. 
\begin{theorem}\label{F5}
    Let $\Xi(t)\in\F_5((t^{-1}))$ be a Laurent series whose coefficients are given by the second-level Paperfolding sequence $(\Pi^{(2)}_n)_{n\ge0}$ as in Definition \ref{npf}. Then, $\Xi(t)$ is a counterexample to $t$-$LC$ over fields of characteristic 5.
\end{theorem}

\noindent Since $\F_p$ is a subfield of $\F_{p^n}$ for any natural number $n$, it is a simple corollary of Theorem \ref{F5} that $t$-$LC$ fails over any finite field of characteristic 5 (see \cite[Section 4]{Faustin} for details). \\

\noindent Not only do the First and Second Level Paperfolding Sequences provide counterexamples to $t$-$LC$, but they also induce counterexamples to $p(t)$-$LC$ for any irreducible polynomial $p(t)$. This is due to the following result by the second named author \cite{My_paper}.

\begin{theorem}\label{my_result}
    Let $p(t)\in\mathbb{K}[t]$ be an irreducible polynomial of degree $m$ and let $l$ be a positive integer. Any counterexample to $t$-$LC$ induces a counterexample to $p(t)$-$LC$ in the following sense: if $\Theta(t)=\sum_{i=1}^\infty a_it^{-i}\in\mathbb{K}((t^{-1}))$ satisfies \[\inf_{\N(t)\in\F_q[t]\backslash\{0\}} |N(t)|\cdot |N(t)|_{t}\cdot \intdist{N(t)\cdot \Theta(t)}=q^{-l},\]  then $\Theta(p(t))=\sum_{i=1}^\infty a_i p(t)^{-i}\in\mathbb{K}((t^{-1}))$ satisfies \[\inf_{\N(t)\in\F_q[t]\backslash\{0\}} |N(t)|\cdot |N(t)|_{p(t)}\cdot \intdist{N(t)\cdot \Theta(t)}=q^{-lm}.\]
\end{theorem}
\noindent Applying Theorem \ref{my_result} to $\Xi(t)$ immediately results in the following corollary:
\begin{cor}\label{F5_cor}
    Let $q$ be a positive power of $5$, let $p(t)\in\F_q[t]$ be an irreducible polynomial and let $(\Pi^{(2)}_i)_{i\ge1}$ be as in Definition \ref{npf}. Then, the Laurent series \[\Xi(p(t)):=\sum_{i=1}^\infty \Pi^{(2)}_{i} p(t)^{-i}\] disproves $p(t)$-$LC$ over $\F_q$. 
\end{cor}

\noindent  The $n^\nth$-Level Paperfolding Sequence is significant not only due to Corollary \ref{F5_cor}, but also because there is strong computational evidence suggesting that, for any field of odd characteristic, there exists some $n\in\N$ such that the Laurent series $\sum_{i=1}^\infty \Pi^{(n)}_i t^{-i}$ provides a counterexample to $t$-LC (and hence $p(t)$-LC by Theorem \ref{my_result}). That is, the $n^\nth$-Level Paperfolding Sequence provides a complete picture of how $p(t)$-$LC$ fails over any finite field of odd characteristic. The method used to find the value of $n$ given the characteristic of the field is explained in the following conjecture, which summarises the expected behaviour of $t$-$LC$ in the cases discussed thus far.
\begin{conjecture}\label{main_conj}
    Let $p$ be an odd prime and let $n\in\N$ be the least natural number such that $p\not\equiv1\mod 2^{n+1}$. Additionally, let $q$ be any positive power of $p$. Then, the Laurent series $\Theta(t)=\sum_{i=1}^\infty \Pi^{(n)}_i t^{-i}\in\F_q((t^{-1}))$ is a counterexample to $t$-LC over $\F_q$. Specifically, $\Theta(t)$ satisfies \begin{equation*}
        \inf_{N(t)\in\F_q[t]\backslash\{0\}} |N(t)|\cdot |N(t)|_t\cdot  |\langle N(t)\cdot\Theta(t)\rangle|=q^{-2^n}.
    \end{equation*}
\end{conjecture}

\noindent Until now, the only evidence towards Conjecture \ref{main_conj} was provided by Adiceam, Nesharim and Lunnon in \cite{Faustin}. To remedy this, the method used to prove Theorem \ref{F5} is also applied to the Paperfolding Laurent series $\Upsilon(t)$ to disprove $t$-$LC$ in fields of characteristic 7 and 11. Combined with Theorem \ref{my_result}, this yields the following result:
\begin{theorem}\label{F7}
    Let $q$ be a positive power of $7$ or $11$, let $p(t)\in\F_q[t]$ be an irreducible polynomial and let $(\Pi^{(1)}_i)_{i\ge1}$ be as in Definition \ref{npf}. Then the Laurent series \[\Upsilon(p(t)):=\sum_{i=1}^\infty \Pi^{(1)}_{i}p(t)^{-i}\] is a counterexample to $p(t)$-$LC$ over $\F_q$. 
\end{theorem}

\noindent The proofs of Theorems \ref{F5} and \ref{F7} are achieved by rephrasing $t$-$LC$ in terms of the so-called \textit{Number Wall} of a sequence, which is defined in Section \ref{section2}. For the purposes of this introduction, the number wall of a one dimensional sequence $S$ is a two dimensional sequence derived entirely from $S$. The core idea used to prove Theorem \ref{F5} and Theorem \ref{F7} is similar to that used by Adiceam, Nesharim and Lunnon in \cite{Faustin}\footnote{The implementation of their algorithm can be found at \cite{Fred_Code} and \cite{erez_code}.}. The number wall is shown to be the limiting sequence of a morphism applied to a finite set of tiles. Both the morphism and the tiles are provided explicitly. Theoretically, the algorithm used in \cite{Faustin} could be used to obtain Theorems \ref{F5} and \ref{F7}. However, the time it would take to do so is not manageable, in any way.\\

\noindent In this paper, two key improvements are made that allow for the proof of Theorems \ref{F5} and \ref{F7} with a much shorter runtime. Firstly, the shape of the tiles used has been changed from \cite{Faustin}, decreasing the amount required to cover the number wall of the First Level Paperfolding sequence over $\F_3$ by a large order of magnitude. Furthermore, a new tiling algorithm has been created that foregoes the need to generate large portions of the number wall under consideration, allowing for savings in both time and memory. This is the main contribution to the computing aspect of the new and growing theory of number walls. The authors believe this algorithm will be needed for future developments occurring in this area of research. This is discussed further in Section \ref{sect:end}. See \cite{code} for the codebase. \\

\noindent The paper is organised as follows. Section \ref{section2} introduces the concept of a number wall and provides the key results needed for the proof of Theorems \ref{F5} and \ref{F7}. Central to these results is the theory of automatic sequences: Section \ref{sect:auto} focuses on this and discusses how the automaticity of a sequence is reflected in its number wall. The proof of both Theorem \ref{F5} and Theorem \ref{F7} is computer assisted and the algorithm used is provided in Section \ref{sect:alg}, with details of the implementation of this code into Python found in Appendix \hyperref[sect:imp]{A}. Finally, further open problems and conjectures are stated in Section \ref{sect:end}.

\subsubsection*{Acknowledgements} \noindent The second named author is grateful to his supervisor Faustin Adiceam for his insight, support and supervision for the duration of this project. The same author acknowledges the financial support of the Heilbronn Institute. Both authors would like to thank the referees for their diligent and thorough feedback, and Fred Lunnon for his correction of Conjecture \ref{main_conj}. Finally, the authors would like to acknowledge the assistance given by Research IT and the use of the Computational Shared Facility at The University of Manchester.

\section{The Correspondence Between $p(t)$-$LC$ and Number Walls}\label{section2}
\noindent The following definition is central to the study of number walls.

\begin{definition} \label{Toe}
A matrix $(s_{i,j})$ for $0\le i\le n, 0\le j \le m$ is called \textbf{Toeplitz} if all the entries on a diagonal are equal. Equivalently, $s_{i,j}=s_{i+1,j+1}$ for any $n\in\mathbb{N}$ such that this entry is defined. 
\end{definition}
\noindent Given a doubly infinite sequence $S:= (s_i)_{i\in\mathbb{Z}}$, natural numbers $m$ and $v$, and an integer $n$, define the $(m+1)\times (v+1)$ \textbf{Toeplitz matrix} $T_S(n, m, v):= (s_{i-j+n})_{0\le i \le m, 0\le j \le v}$ as\\
\begin{align*}T_S(n,m,v):=\begin{pmatrix}
s_n & s_{n+1} & \dots & s_{n+v}\\
s_{n-1} & s_{n} & \dots & s_{n+v-1}\\
\vdots &&& \vdots\\
s_{n-m} & s_{n-m+1} & \dots & s_{n-m+v} 
\end{pmatrix}.\end{align*} If $v=m$, this is abbreviated to $T_S(n,m)$. The Laurent series $\Theta(t)=\sum^\infty_{i=1}s_it^{-1}\in\F_{q}((t^{-1}))$ is identified with the sequence $S=(s_i)_{i\ge1}$. Accordingly, let $T_\Theta(n,m,v)=T_S(n,m,v)$.

\begin{definition}\label{nw}
Let $S=(s_i)_{i\in\mathbb{Z}}$ be a doubly infinite sequence over a finite field $\mathbb{F}_q$. The \textbf{number wall} of the sequence $S$ is defined as the two dimensional array of numbers $W_q(S)=(W_{q,m,n}(S))_{n,m\in\mathbb{Z}}$ with \begin{equation*}
    W_{q,m,n}(S)=\begin{cases}\det(T_S(n,m)) &\textup{ if } m\ge0,\\
    1 & \textup{ if } m=-1,\\
    0 & \textup{ if } m<-1. \end{cases}
\end{equation*}  
\end{definition} 
\noindent In keeping with standard matrix notation, $m$ increases as the rows go down the page and $n$ increases from left to right. The remainder of this section only mentions results crucial for this paper. For a more comprehensive look at number walls, see \cite[Section 3]{Faustin}, \cite[pp. 85--89]{Conway1995TheBO}, \cite{Lunnon_2009,My_paper} and \cite{numbwall}. \\

\noindent A key feature of number walls is that the zero entries can only appear in specific shapes:
\begin{theorem}[Square Window Theorem]\label{window}
Zero entries in a Number Wall can only occur within windows; that is, within
square regions with horizontal and vertical edges.
\end{theorem}
\begin{proof}
See \cite[p. 9]{numbwall}.
\end{proof}

\noindent Theorem \ref{window} implies that the border of a window is always the boundary of a square with nonzero entries. This motivates the following definition: 
\begin{definition}
The entries of a number wall surrounding a window are referred to as the \textbf{inner frame}. The entries surrounding the inner frame form the \textbf{outer frame}. 
\end{definition}
\noindent The entries of the inner frame are extremely well structured:
\begin{theorem}\label{ratio ratio}
The inner frame of a window with side length $l\ge1$ is comprised of 4 geometric sequences. These are along the top, left, right and bottom edges and they have ratios $P,Q,R$ and $S$ respectively with origins at the top left and bottom right. Furthermore, these ratios satisfy the relation \[\frac{PS}{QR}=(-1)^{l}.\]
\end{theorem}
\begin{proof}
   See \cite[p11]{numbwall}.
\end{proof}
\noindent See Figure \ref{basicwondow} for an example of a window of side length $l$. For $i\in\{0,\dots,l+1\}$, the inner and outer frames are labelled by the entries $A_i,B_i,C_i,D_i$ and $E_i,F_i,G_i,H_i$ respectively. The ratios of the geometric sequences comprising the inner frame are labelled as $P,Q,R$ and $S$.
\begin{figure}[H]
\centering
\includegraphics[width=0.5\textwidth]{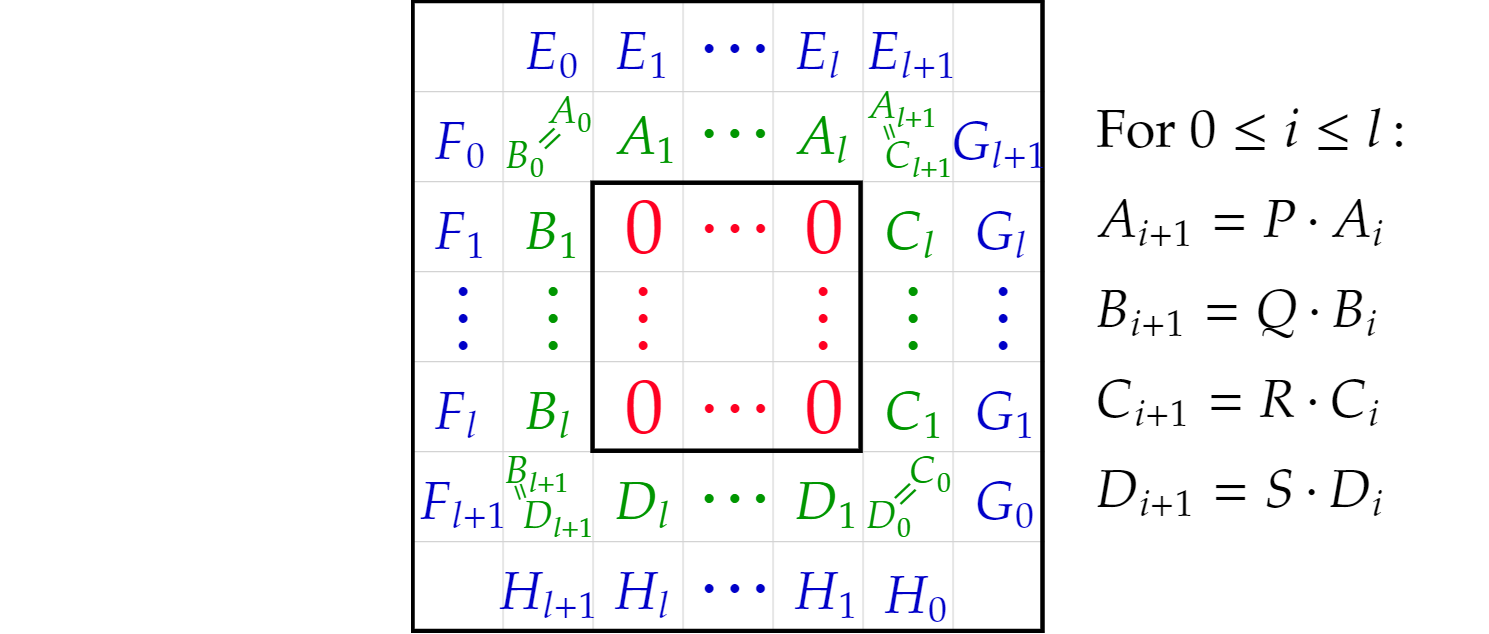}
\caption{Illustration of a window in a number wall. The window, inner frame and outer frame are in red, green and blue, respectively.}
\label{basicwondow}
\end{figure}

\noindent Calculating a number wall from its definition is a computationally exhausting task. The following theorem gives a simple and far more efficient way to calculate the $m^\nth$ row using the previous rows.
\begin{theorem}[Frame Constraints] \label{FC}
Given a doubly infinite sequence $S=(s_{i})_{i\in\mathbb{Z}}$ over the field $\K$, the number wall $W_{\K}(S):=(W_{m,n})_{n,m\in\mathbb{Z}}$ can be generated by a recurrence in row $m\in\mathbb{Z}$ in terms of the
previous rows. More precisely, with the notation of Figure \ref{basicwondow},\begin{equation*}
    W_{m,n}=\begin{cases}
    \begin{aligned}
    &0 &&\textup{ if } m<-1 \textup{ or if } (m,n)\textup{ is within a window;}&\\
    &1 &&\textup{ if } m=-1;&\\
    &s_n &&\textup{ if } m=0;&\\
    &\frac{W_{m-1,n}^2-W_{m-1,n-1}W_{m-1n+1}}{W_{m-2,n}} &&\textup{ if } m>0\textup{ and }W_{m-2,n}\neq0;&\\
    &D_k=\frac{(-1)^{l\cdot k}B_kC_k}{A_k}& &\textup{ if } m>0\textup{ and }W_{m-2,n}=0=W_{m-1,n};&\\
    &H_k=\frac{\frac{QE_k}{A_k}+(-1)^{k}\frac{PF1_k}{B_k}-(-1)^k\frac{SG_k}{C_k}}{R/D_k}& &\textup{ if } m>0 \textup{ and } W_{n,m-2}=0\neq W_{n,m-1}.&
    \end{aligned}
    \end{cases}
\end{equation*}
\end{theorem}
\begin{proof}
    See \cite[p11]{numbwall}
\end{proof}
\noindent The value of $k$ above is found in the natural way from the value of $m,n$ and the side length $l$. The final three equations above are known as the \textbf{First}, \textbf{Second} and \textbf{Third Frame Constraint Equations}. These allow the number wall of a sequence to be considered independently of its original definition in terms of Toeplitz matrices. It is simple to see that a finite sequence of length $r$ generates a number wall in the shape of an isosceles
triangle with depth $\left\lfloor\frac{r-1}{2}\right\rfloor$. If $S$ is a finite sequence in $\F_q$, then denote the finite number wall of $S$ as $W_q(S)$. To make number walls visually accessible, each entry is given a unique colour depending on its value (See Figure \ref{smallrand}).\begin{figure}[H]
\centering
\includegraphics[width=0.3\textwidth]{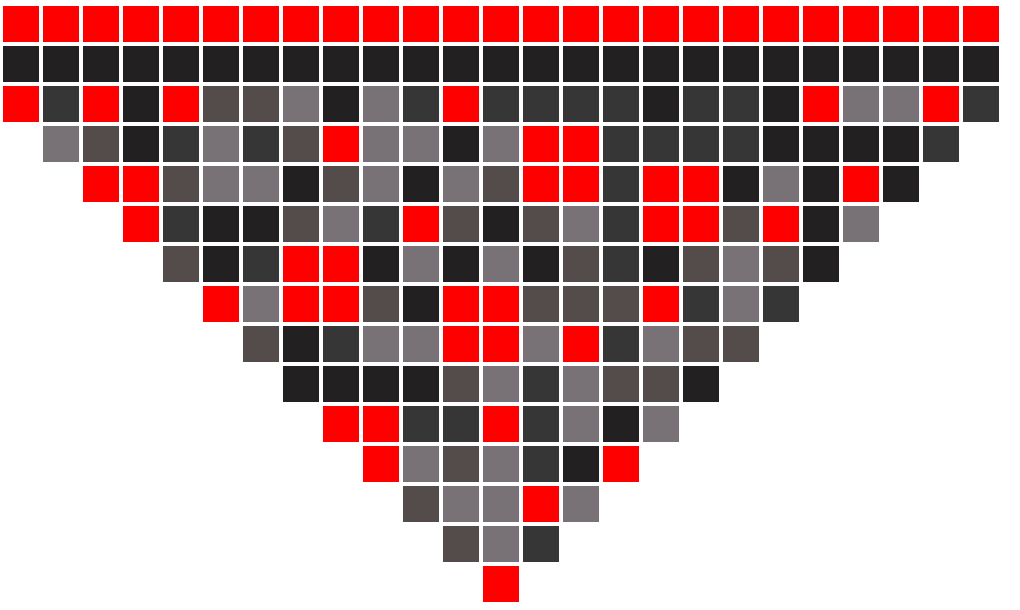}
\caption{The number wall of a sequence of length 25 generated uniformly and randomly over $\F_5$. The zero entries are coloured in red, with the nonzero values assigned a shade of grey, with 1 being the darkest and 4 being the lightest. The top row (red) has index $-2$. It is included as it initiates the recurrence relations given by Theorem \ref{FC} that are used to generate this image.}
\label{smallrand}
\end{figure}
\noindent The following result from \cite[p5]{Faustin} rephrases the $t$-adic Littlewood conjecture in terms of number walls.
\begin{lemma} \label{growcor}
Let $\Theta(t)=\sum^\infty_{i=1}s_it^{-i}\in\mathbb{F}_q((t^{-1}))$ be a Laurent series and $S=(s_i)_{i\in\mathbb{N}}$ be the sequence of its coefficients. Then, $\Theta(t)$ is a counterexample to $t$-LC if and only if there exists an $l$ in the natural numbers such that the number wall of $(s_i)_{i\ge0}$ has no window of size larger than $l$.
\end{lemma}
 \noindent Combining Lemma \ref{growcor} with Theorem \ref{my_result} reduces the proof of Theorems \ref{F5} and \ref{F7} to proving that the number wall of the First and the Second Level Paperfolding Sequences, over their respective finite fields, have bounded window size. These sequences are extremely structured, and the proof of Theorems \ref{F5} and \ref{F7} is achieved by proving their respective number walls inherit this structure. To formalise this idea, Section \ref{sect:auto} introduces the concept of automatic sequences. 
\section{Detecting Automaticity in a Number Wall}\label{sect:auto}
\noindent This section only covers the basic properties of automatic sequences needed for this paper. See \cite{Allouche_Shallit_2003} for more on this theory.\\

\noindent Let $\Sigma$ be a finite alphabet and let $\Sigma^*$ be the set of all (finite and infinite) words made up from $\Sigma$, including the empty word. For finite words $x,y\in\Sigma^*$, let $xy$ denote the concatenation of $y$ to the right side of $x$. Given another finite alphabet $\Delta$, a function $\phi:\Sigma^*\to\Delta^*$ is called a \textbf{morphism} if $\phi(xy)=\phi(x)\phi(y)$ for any finite words $x,y\in\Sigma^*$. Furthermore, let $L(w)$ denote the length of the finite word $w\in\Sigma^*$. If $k$ is a natural number, then $\phi$ is a \textbf{uniform} $k$-\textbf{morphism} if $L(\phi(w))=k\cdot L(w)$ for any finite word $w$ in $\Sigma^*$.\\

\noindent Let $\phi:\Sigma^*\to\Sigma^*$ be a morphism. An element of $\Sigma$ is called $\phi$-\textbf{prolongable} if it is the first letter of its own image under $\phi$. If a sequence $S=(s_n)_{n\ge0}$ is defined over $\Sigma$ and satisfies $\phi(S)=S$, then $S$ is a \textbf{fixed point} of $\phi$. Trivially, every sequence is the fixed point of the identity 1-morphism. However, ignoring this uninteresting case, one has that a sequence $S=(s_n)_{n\ge0}$ is a fixed point of a morphism $\phi$ if and only if $s_0$ is $\phi$-prolongable and $$S=\phi^\omega(s_0):=\lim_{i\to\infty}\phi^i(s_0),$$ where $\phi^i(s_0):=\phi(\phi^{i-1}(s_0))$ is defined iteratively in the natural way.\\

\noindent For a finite alphabet $\Delta$ which is distinct from $\Sigma$ and a natural number $d$, a $d$-\textbf{coding} is a uniform $d$-morphism $\tau:\Sigma\to\Delta^d$, where $\Delta^d$ is the set of all finite words of length $d\in\N$ made up from $\Delta$. It is not necessary that $\tau$ be injective. The notation $\tau(S)$ is shorthand for applying $\tau$ to each element of the sequence $S$. The following concept is at the heart of the proof of Theorems \ref{F5} and \ref{F7}.
\begin{definition}
    Let $k$ be a natural number. An infinite sequence $S=(s_n)_{n\ge0}$ is called $k$-\textbf{automatic} if for each $n\ge0$, $s_n$ can be derived from a finite state automaton which takes the base-$k$ digits of $n$ as an input.  
\end{definition}
\noindent For the purposes of this paper, the following theorem of Cobham \cite{Cobham} (rephrased in modern language in \cite[Theorem 6.3.2]{Allouche_Shallit_2003}) can be seen as an alternative definition for automatic sequences:
\begin{theorem}[Cobham's Theorem]
    A sequence $S$ is \textbf{automatic} if and only there exist natural numbers $k$ and $d$ such that $S$ is the image by a $d$-coding of a fixed point of a uniform k-morphism.
\end{theorem}
\noindent The following example is central to this paper. The notation used below differs from what is used in the previous definitions, since it specifically refers to the so-called \textit{First Level Paperfolding Sequence} which will appear numerous times throughout this work. Hence, this specialised notation is introduced and fixed from now on.
\begin{example}\label{PF1DEF}
    Let $\Gamma=\{A,B,C,D\}$ be an alphabet and define the 2-morphism $\psi:\Gamma\to\Gamma^*$ as \begin{align}
   \label{PF1 sub rules}  &\psi(A)=AB& &\psi(B)=CB& &\psi(C)=AD& &\psi(D)=CD.&
\end{align}Let $\Delta=\{0,1\}$. Then the First Level Paperfolding sequence can be constructed as $\pi(\psi^\omega(A))$, where $\pi:\Gamma\to\Delta$ is a 1-coding defined by\begin{align}
    &\pi(A)=0& &\pi(B)=0& &\pi(C)=1& &\pi(D)=1.&
\end{align}
\end{example}
\noindent The choice of $d$-coding generating a particular $k$-automatic sequence is not unique. For example, if $S$ is a $k$-automatic sequence generated by a $k$-morphism $\phi$ and $\tau$ is a $d$-coding on a finite alphabet $\Sigma=\{\sigma_0,\dots,\sigma_n\}$, one can define a $kd$-coding called the $k$-\textbf{compression} of $\tau'$ by setting $\tau'(\sigma_i)=\tau(\phi(\sigma_i))$. Applying the $k$-compression $n$ times will be referred to as taking the $k^n$-compression.
\begin{example}\label{PF1-4-comp}
    The $2$-compression of the First Level Paperfolding sequence is given by the same 2-morphism $\psi$ as in (\ref{PF1 sub rules}) and by the $2$-coding $\pi'$ defined as\begin{align*}
    &\pi'(A)=00& &\pi'(B)=10& &\pi'(C)=01& &\pi'(D)=11.&
\end{align*}
    Applying the 2-compression again returns the 4-compression, which is a $4$-coding: 
    \begin{align*}
    &\pi''(A)=0010& &\pi''(B)=0110& &\pi''(C)=0011& &\pi''(D)=0111.&
\end{align*}
\end{example}
\subsection{Two-Dimensional automatic Sequences}
\noindent The content of this subsection is adapted from \cite[Chapter 14]{Allouche_Shallit_2003} by Allouche and Shallit. For more information on two-dimensional automatic sequences, see \cite{Salon1989} and \cite{Salon2}.\\

\noindent Automatic sequences have a natural generalisation to higher dimensions. For the purposes of proving Theorems \ref{F5} and \ref{F7}, only two-dimensional automatic sequences are required and are thus defined below. \\

\noindent Let $\Sigma$ be a finite alphabet. A two-dimensional word $\mathbf{w}$ made up from $\Sigma$ is defined as a matrix $$\textbf{w}=(\sigma_{m,n})_{0\le m \le M, 0\le n \le N} ~~~~\text{ where }~~~~ M,N\ge0.$$ In agreement with standard matrix terminology, $\textbf{w}_{m,n}=\sigma_{m,n}\in\Sigma$ denotes the entry of $\mathbf{w}$ in column $n$ and row $m$. \\

\noindent Let $\Sigma^*_2$ denote the set of (finite and infinite) two-dimensional words made up from $\Sigma$. Given another (potentially identical) finite alphabet $\Delta$ and $k,l\in\N$, a \textbf{two-dimensional uniform $[k,l]$-morphism }$\phi:\Sigma^*_2\to\Delta^*_2$ is a function that satisfies the following two properties: \begin{enumerate}
    \item For any letter $\sigma\in\Sigma$, $\phi(\sigma)$ returns a $k\times l$ matrix in $\Sigma^*_2$.
    \item Let $M,N\in\N$ and let $\textbf{w}=(\sigma_{m,n})_{0\le m \le M,0\le n \le N},\in\Sigma_2^*$ be a finite two-dimensional word. Then, the function $\phi$ satisfies \[\phi\begin{pmatrix}
        \sigma_{0,0}&\sigma_{0,1}&\cdots &\sigma_{0,N}\\\vdots&&&\vdots\\\sigma_{M,0}&\sigma_{M,1}&\cdots&\sigma_{M,N}
    \end{pmatrix}=~\begin{pmatrix}
        \phi(\sigma_{0,0})&\phi(\sigma_{0,1})&\cdots &\phi(\sigma_{0,N})\\\vdots&&&\vdots\\\phi(\sigma_{M,0})&\phi(\sigma_{M,1})&\cdots&\phi(\sigma_{M,N})
    \end{pmatrix}\]
\end{enumerate} If $\Delta=\Sigma$, then a letter $\sigma_0\in\Sigma$ is called $\phi$\textbf{-prolongable} if $\phi(\sigma_0)_{0,0}=\sigma_0$. Just as in the one-dimensional case, define $\phi^k(\sigma_0)$ iteratively as $\phi^k(\sigma_0)=\phi(\phi^{k-1}(\sigma_0))$ and $\phi^0(\sigma_0)=\sigma_0$. Then, the limiting sequence $\phi^\omega(\sigma_0)=\lim_{k\to\infty}\phi^k(\sigma_0)$ exists if and only if $\sigma_0$ is $\phi$-prolongable. 
\begin{example}\label{2dim}
    Let $\Sigma=\{0,1\}$ and define the two-dimensional $[2,2]$-morphism $\phi:\Sigma^*_2\to\Sigma^*_2$ as \begin{align*}
        &\phi(0)=\begin{matrix}
            0&1\\1&0
        \end{matrix}& &\phi(1)=\begin{matrix}
            1&0\\0&1
        \end{matrix}&
    \end{align*}
\noindent Then, applying $\phi$ repeatedly to the $\phi$-prolongable element $0\in\Sigma$ gives \begin{align*}
    &\phi^0(0)=0& &\phi^1(0)=\begin{matrix}
            0&1\\1&0
        \end{matrix}& &\phi^2(0)=\begin{matrix}
            0&1&1&0\\1&0&0&1\\1&0&0&1\\0&1&1&0
        \end{matrix}&&\phi^3(0)=\begin{matrix}
            0&1&1&0&1&0&0&1\\1&0&0&1&0&1&1&0\\1&0&0&1&0&1&1&0\\0&1&1&0&1&0&0&1\\1&0&0&1&0&1&1&0\\0&1&1&0&1&0&0&1\\0&1&1&0&1&0&0&1\\1&0&0&1&0&1&1&0
        \end{matrix}&
\end{align*}\end{example}
\noindent The two-dimensional version of a $d$-coding is now defined.
\begin{definition}
    Let $\Delta$ and $\Sigma$ be distinct finite alphabets and let $k$ and $l$ be natural numbers. Then, a two-dimensional $[k,l]$-\textbf{coding} is a $[k,l]$-morphism $\tau:\Sigma_2^*\to\Delta_2^*$. 
\end{definition} \begin{example}
\noindent Let $\Delta=\{a,b\}$ and $\phi$ be as in Example \ref{2dim}. One example of a two-dimensional $[1,2]$-coding $\tau:\Sigma\to\Delta$ is defined by \begin{align*}
    &\tau(0)=~\begin{matrix}a&a\end{matrix}& &\tau(1)=~\begin{matrix}b&b\end{matrix}& &\tau\left(\phi^2(0)\right)=\begin{matrix}
            a&a&b&b&b&b&a&a\\b&b&a&a&a&a&b&b\\b&b&a&a&a&a&b&b\\a&a&b&b&b&b&a&a
        \end{matrix}&
\end{align*} 
\end{example}
\noindent The concept of a two-dimensional 2-automatic sequence is defined similarly to the one-dimensional case, using a generalised version of Cobham's Theorem by Allouche and Shallit \cite[Theorem 14.2.3]{Allouche_Shallit_2003}.\begin{definition}[Cobham's Theorem for Two-Dimensional Sequences:]
    A two dimensional sequence $S$ is \textbf{automatic} if and only if there exist natural numbers $k,l,k',l'$ such that $S$ is the image (under a $[k',l']$-coding) of a fixed point of a two-dimensional uniform $[k,l]$-morphism.
\end{definition}

\subsection{The Number Wall of an Automatic Sequence}

\noindent Let $\phi:\Sigma^*\to\Sigma^*$ be a $k$-morphism on the finite alphabet $\Sigma$, and let $\sigma_0\in\Sigma$ be $\phi$-prolongable. Define $S=\tau(\phi^\omega(\sigma_0))$ for a one dimensional $d$-coding $\tau:\Sigma\to\F_q^d$, where $d\in\N$. If $k$ is even, it may be assumed without loss of generality that $d$ is also even upon replacing $\tau(\sigma)$ with its $k$-compression. \\

\noindent The image of each letter $\sigma\in\Sigma$ under $\tau$ forms a finite sequence with which is associated a finite number wall as in Figure \ref{smallrand}. Rows with negative index are included to create horizontal symmetry, resulting in the following shape: \begin{figure}[H]
\begin{center}
    \includegraphics[scale=0.5]{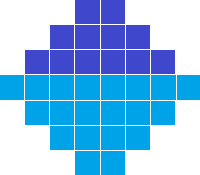}
    \caption{\label{tile_fig}The shape of a finite number wall generated by $\tau(\sigma)$ when $L(\tau(\sigma))=8$. Each block represents an entry, with the dark blue representing the rows with index $m\le -1$ and the longest row being the sequence given by $\tau(\sigma)$.}
    \vspace{-1.5em}
\end{center}
\end{figure} 
\noindent Blocks of the number wall in the shape of Figure \ref{tile_fig} are called \textbf{tiles}. Figure \ref{nw_tiles} depicts how a number wall is split into tiles. When $k$ is odd and $\tau(\sigma)$ has odd length for any tile $\sigma\in\Sigma$, the number wall is split into tiles as depicted in the right image of Figure \ref{nw_tiles}. This paper only deals with $2$-morphisms, and hence the number wall will always be split into tiles as depicted by the left image of Figure \ref{nw_tiles}. 
\begin{figure}[H]
\begin{center}
    \includegraphics[scale=0.20]{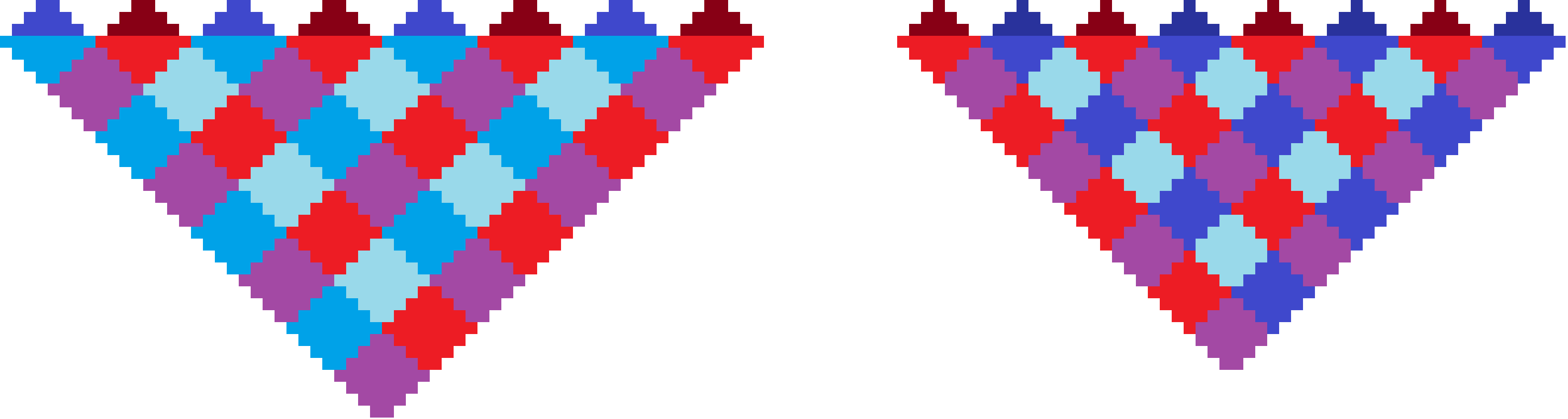}
    \caption{\label{nw_tiles}A large potion of the number wall split into tiles depending on if $\tau(\sigma)$ has odd or even length. The dark red and blue show where the row index becomes negative. The remaining colours serve no other purpose than to distinguish one tile from another.}
    \vspace{-1.5em}
\end{center}
\end{figure} 

\noindent Just as in Figure \ref{tile_fig}, the rows of negative index are included for horizontal symmetry. However, by including this seemingly redundant information, one is now able to consider the number wall as a two-dimensional sequence of tiles, $\mathcal{T}=(T_{m,n})_{m,n\in\N}$, in the following sense: \begin{figure}[H]
    \centering
    \includegraphics[width=1\linewidth]{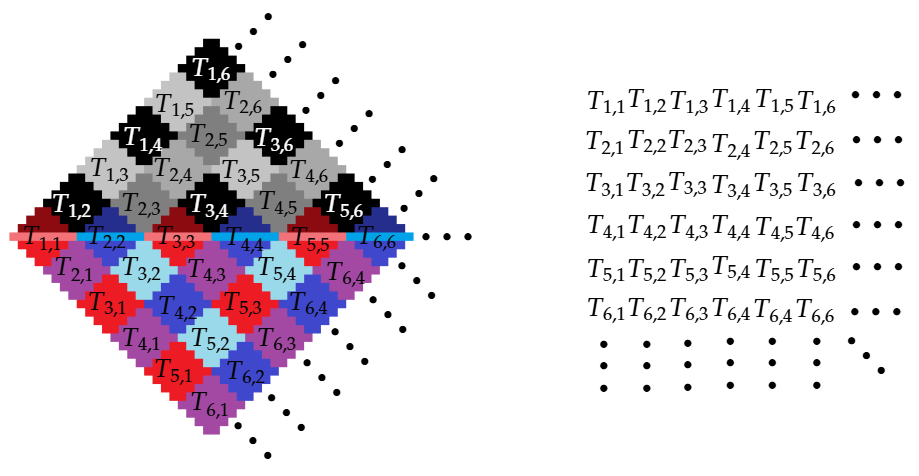}
    \caption{\textbf{Left:} The number wall of a sequence split into tiles. Any tiles that are entirely contained on rows of negative index are coloured in black and grey. The lighter row in the middle denotes the zeroth row of the number wall. The remaining colours serve no other purpose than to distinguish one tile from another. \textbf{Right:} The two dimensional sequence $\mathcal{T}$ that is identified with the number wall.  }
    \label{auto_nw}
\end{figure} 

\noindent For the purposes of this paper, the entries of number walls are in finite fields. As a consequence, there are only finitely many tiles of any given size and hence the sequence $\mathcal{T}$ is defined over a finite alphabet.\\

\noindent Theorems \ref{F5} and \ref{F7} are attained by proving that the number wall of the First and Second Level Paperfolding Sequences (over their respective finite fields) are themselves two-dimensional 2-automatic sequences, in the sense of Figure \ref{auto_nw}. Therefore, one needs to define the function used to transform the sequence $\mathcal{T}$ into the shape of a number wall. However, it is clear from their rectangular shape that $[k,l]$-codings are not sufficient for this task on their own. To remedy this, the following transformation is introduced that takes a two-dimensional infinite sequence in $\F_q^{\N\times\N}$ and returns a two-dimensional array in the shape of a number wall. \begin{definition}
    Define $\mathcal{P}:\F_q^{\N\times\N}\to\F_q^{\Z\times\N}$ as the function \[\mathcal{P}((s_{m,n})_{m,n\in\N})=(b_{m,n})_{m\in\Z, n\ge|m|}~~~ \text{where} ~~~b_{m,n}=s_{m+n,\left\lfloor\frac{n-m}{2}\right\rfloor}.\]
\noindent In other words, $\mathcal{P}$ permutes the entry in row $m$ and column $n$ of the input two-dimensional sequence to row $-n+\left\lfloor\frac{m}{2}\right\rfloor$ and column $n+\left\lceil\frac{m}{2}\right\rceil$. This is illustrated in the following example.
\end{definition}
\begin{example}
    Let $(s_{m,n})_{m,n,\in\N}$ be a two dimensional sequence. Then, \begin{equation*}
        \mathcal{P}\begin{pmatrix}
            s_{0,0}& s_{0,1} &s_{0,2} &s_{0,3} & \cdots\\
            s_{1,0}& s_{1,1} &s_{1,2} &s_{1,3} & \cdots\\
            s_{2,0}& s_{2,1} &s_{2,2} &s_{2,3} & \cdots\\
            s_{3,0}& s_{3,1} &s_{3,2} &s_{3,3} & \cdots\\
            s_{4,0}& s_{4,1} &s_{4,2} &s_{4,3} & \cdots\\
            s_{5,0}& s_{5,1} &s_{5,2} &s_{5,3} & \cdots\\
            s_{6,0}& s_{6,1} &s_{6,2} &s_{6,3} & \cdots\\
            s_{7,0}& s_{7,1} &s_{7,2} &s_{7,3} & \cdots\\
            \vdots&\vdots&\vdots&\vdots&\ddots
        \end{pmatrix}=\begin{matrix}
            &&&&\iddots&\iddots\\
            &&&s_{0,3}&s_{1,3}&\iddots&\iddots\\
            &&s_{0,2}&s_{1,2}&s_{2,3}&s_{3,3}&\iddots&\iddots\\
            &s_{0,1}&s_{1,1}&s_{2,2}&s_{3,2}&s_{4,3}&s_{5,3}&\iddots\\
            s_{0,0}&s_{1,0}&s_{2,1}&s_{3,1}&s_{4,2}&s_{5,2}&s_{6,3}&s_{7,3}&\cdots\\
            &s_{2,0}&s_{3,0}&s_{4,1}&s_{5,1}&s_{6,2}&s_{7,2}&\ddots\\
            &&s_{4,0}&s_{5,0}&s_{6,1}&s_{7,1}&\ddots&\ddots\\
            &&&s_{6,0}&s_{7,0}&\ddots&\ddots\\
            &&&&\ddots&\ddots
        \end{matrix}
    \end{equation*}
\end{example}
\noindent Hence, a number wall is considered to be an automatic sequence in the sense that it is equal to $\mathcal{P}(\tau'(\mathcal{T}))$, where $\tau'$ is a $[k',l']$-coding for some $k',l'\in\N$ and $\mathcal{T}$ is the fixed point of a two-dimensional $[k,l]$-morphism for some $k,l\in\N$. In particular, the number walls of the First-Level (Second-Level, respectively) Paperfolding sequences are shown to be equal to $\mathcal{P}(\tau'(\mathcal{T}))$ where $\tau'$ is a $[8,4]$-coding ($[16,8]$-coding, respectively) and $\mathcal{T}$ is the fixed point of a $[2,2]$-morphism. \\

\noindent In summary, number walls are considered automatic sequences in the following sense. \begin{figure}[H]
    \centering
    \includegraphics[width=1\linewidth]{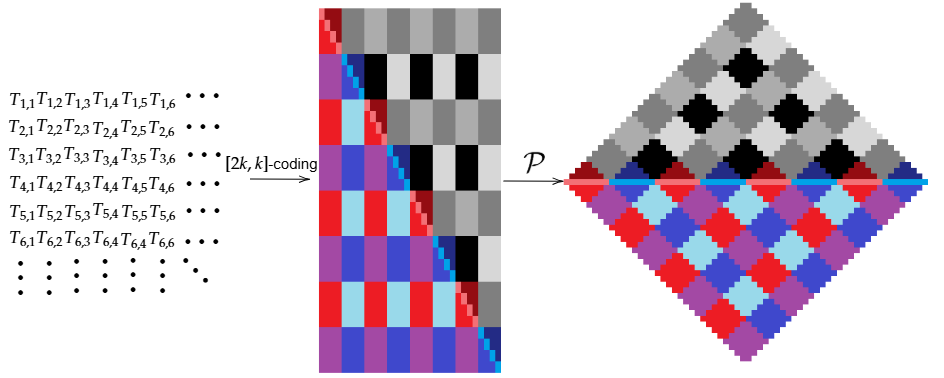}
    \caption{In this paper, a number wall is automatic if it is obtained by applying $\mathcal{P}$ to the image of a fixed point of a two-dimensional $[2,2]$-morphism under a $[2k,k]$-coding.}
    \label{fig_6}
\end{figure}
\noindent Let $k\in\N$, let $T$ be the finite set of tiles that are comprised of $2k^2$ elements over $\F_q$ and let $\Sigma$ be a finite alphabet. In practice, the $[2k,k]$-coding $\tau'$ and the transformation $\mathcal{P}$ are combined into one function $\tau=\mathcal{P}\circ\tau':\Sigma_2^*\to T_2^*$ that takes a two-dimensional sequence over a finite alphabet $\Sigma$ (as in the left-side of Figure \ref{fig_6}) and returns a two-dimensional sequence of tiles (as in the right-side of Figure \ref{fig_6}). This function is called a \textbf{number wall coding}. 
\section{Explicit Morphism for the Number Wall of the First Level Paperfolding Sequence}\label{sect:alg}
\noindent Let $S$ be the First Level Paperfolding Sequence and let $\psi$ be the 2-morphism that generates it (see Example \ref{PF1DEF}). This section details the algorithm used to prove Theorem \ref{F7}. That is, this algorithm finds an ordered finite alphabet $\Sigma=\Sigma_{S,p}$, a $[2,2]$-morphism $\Psi=\Psi_{S,p}:\Sigma_2^*\to\Sigma_2^*$, $\sigma_0\in\Sigma$ that is $\Psi$-prolongable, a finite set of tiles $T=T_{S,p}$ and a number wall coding $\tau=\tau_{S,p}:\Sigma_2^*\to T_2^*$ such that\footnote{Recall here Definition \ref{nw}.} $W_p(S)=\tau(\Psi^\omega(\sigma_0))$. The algorithm is almost identical when working with the Second Level Paperfolding Sequence to prove Theorem \ref{F5}; the minor changes will be detailed at the end of this section. The proof is split into two main steps:\begin{enumerate}
    \item An efficient algorithm uses $\psi$ and the Frame Constraints (Theorem \ref{FC}) to find the alphabet $\Sigma$, the $[2,2]$-morphism $\Psi$, the letter $\sigma_0$, the set of tiles $T$ and the number wall coding $\tau$. 
    \item The two-dimensional sequence defined by ${\tau}({\Psi^\omega}({\sigma_0}))$ is proven to be equal to $W_p(S)$ by verifying that ${\tau}({\Psi^\omega}({\sigma_0}))$ satisfies the Frame Constraints (Theorem \ref{FC}) and that the zeroth row is $S$.
\end{enumerate} 
\noindent Each step is split into two further sub-steps (\hyperref[11]{1.1}, \hyperref[12]{1.2}, \hyperref[21]{2.1}, and \hyperref[22]{2.2} below) and each sub-step is presented as its own algorithm. The steps of each algorithm are written {\fontfamily{cmtt}\selectfont in typewriter font} and any necessary justification of each step is provided below in the normal font.
\subsection*{Part 1: Finding the Parameters}
The algorithm begins with empty lists $\Sigma$ and $T$. As it progresses, letters (tiles, respectively) are added to $\Sigma$ (to $T$, respectively) and the $[2,2]$-morphism $\Psi$ and number wall coding $\tau$ are built. The finite alphabet ${\Sigma}$ will be comprised of letters $\sigma_i$, where $i$ is in a finite set $\{n\in\Z: -1\le n \le N\}$ for some $N<\infty$. Finally, given a finite word $w\in\F_q^*$, one defines the number wall of $w$ in the natural way by considering $w$ as a finite sequence of letters. This is denoted $W_q(w)$.
\subsubsection*{Algorithm 1.1: Initial Conditions}\label{11}
Let $\Gamma$, $\psi$ be as they were in Example \ref{PF1DEF} and define $\pi$ as the 8-compression of the 1-coding from the same example. From now on, $\# T$ denotes the cardinality of a finite ordered list $T$.\\

\noindent {\fontfamily{cmtt}\selectfont\textbf{Step 1.1.1:} The tile whose bottom row has index $-1$ is appended to ${T}$ and $\sigma_{-1}$ is appended to ${\Sigma}$.\\ Define $\tau(\sigma_{-1})$ as this tile. Similarly, the tile of all zeros is appended to $T$, $\sigma_0$ is appended \\to $\Sigma$ and $\tau(\sigma_0)$ is defined as this tile. These tiles are depicted in Figure \ref{zero_tile}.}
\begin{figure}[H]
    \centering
    \includegraphics[width=0.65\linewidth]{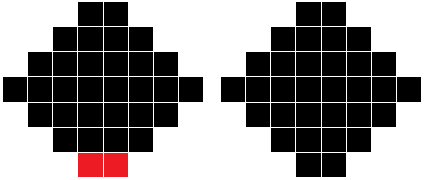}
    \caption{\label{zero_tile}\textbf{Left:} The tile whose bottom row has index $-1$. The red entries are equal to 1, and the black entries are 0. \text{Right:} The tile of all zeros, with the same colour scheme.}
\end{figure}
\noindent {\fontfamily{cmtt}\selectfont\textbf{Step 1.1.2:} Recall the finite alphabet $\Gamma$ and the 1-coding $\pi$ from Example \ref{PF1 sub rules}. For each $\gamma\in\Gamma$, \\calculate the tile given by $W_p(\pi(\gamma))$ and append it to ${T}$. Each time a tile is appended to ${T}$, \\append $\sigma_{m(T)}$ to ${\Sigma}$ where $m(T)=\#{T}-2$. Additionally, define ${\tau}(\sigma_{m(T)})=W_p(\pi(\gamma))$.}\\

\noindent Note that a new letter is appended to ${\Sigma}$ whenever a new tile is discovered. Therefore, ${\tau}:\Sigma\to T$ is a bijection and hence its inverse is well defined. \\

\noindent The image of $\sigma_i\in{\Sigma}$ under ${\Psi}$ is a $2\times 2$ matrix $\begin{matrix}\sigma_{i_w}&\sigma_{i_n}\\\sigma_{i_s}&\sigma_{i_e}\end{matrix}$. Here, each entry of the 4-tuple has been associated to a cardinal direction (east, north, west, south) in the natural way.  For any 4-tuple of tiles that are arranged as depicted in Figure \ref{2-morph}, the \textbf{scaffolding} of the southern tile is defined as the 3-tuple given by the eastern, northern and western tiles. 
\begin{figure}[H]
    \centering
    \includegraphics[width=0.75\linewidth]{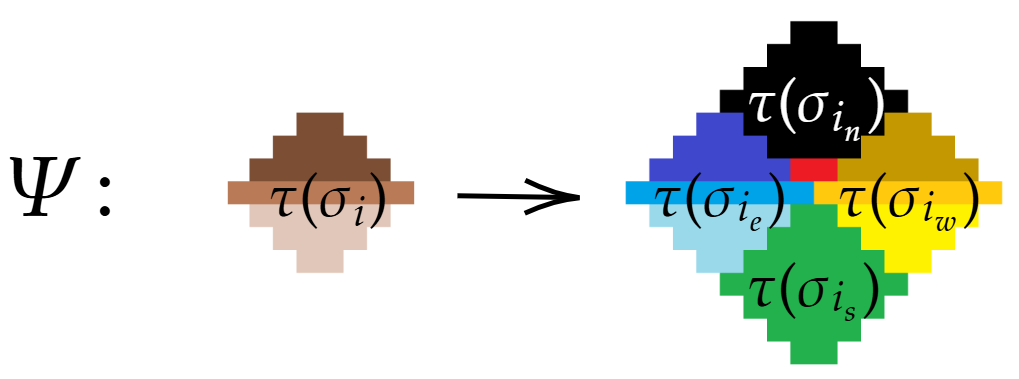}
    \caption{\label{2-morph}An illustration of $\Psi$ acting on a tile $\tau(\sigma_i)$ which contains part of the zeroth row of the number wall. The black and red tile $\tau(\sigma_{-1})$ is the same as that from Figure \ref{zero_tile}. The middle row of the brown, blue and yellow tiles is made up from the zeroth row of the number wall. The lighter (darker, respectively) portions of these tiles have positive (negative, respectively) row index. }
    
\end{figure}
\noindent For $\gamma\in\Gamma$ and for $i\in\{0,1\}$, define $\psi(\gamma)_i$ be the $i^\nth$ letter of $\psi(\gamma).$\\

\noindent {\fontfamily{cmtt}\selectfont\textbf{Step 1.1.3:} Create an empty set $Scaf$. For each $\gamma\in\Gamma$, split $W_p(\pi(\psi(\gamma)))$ into 4 tiles as de-\\picted in Figure \ref{2-morph}. The southern tile, denoted $t_s$, is appended to ${T}$ and $\sigma_{m(T)}$ is appended to \\$\Sigma$. Append the scaffolding of $t_s$ to $Scaf$, and then define $${\Psi}(\tau^{-1}(W_p(\pi(\gamma)))):=\begin{matrix}{\tau}^{-1}(W_p(\pi(\psi_{0}(\gamma))))&~\sigma_{-1}\\\sigma_{m(T)}&{\tau}^{-1}(W_p(\pi(\psi_{0}(\gamma))&\end{matrix}.$$}

\noindent {\fontfamily{cmtt}\selectfont\textbf{Step 1.1.4:} For completeness, define $\Psi(\sigma_{-1})=~\begin{matrix}\sigma_0&\sigma_0\\\sigma_{-1}&\sigma_{0}\end{matrix}$ and $\Psi(\sigma_{0})=~\begin{matrix}\sigma_0&\sigma_0\\\sigma_0&\sigma_{0}\end{matrix}$.}\\

\noindent {\fontfamily{cmtt}\selectfont\textbf{Step 1.1.5:} Return $Scaf$, ${T}$, ${\Sigma}$, ${\tau}$ and ${\Psi}$}\\

\noindent At the end of Part 1.1, one has found all the tiles that make up the zeroth row of the number wall and the functions ${\Psi}$ and ${\tau}$ have been defined on these tiles. Furthermore, four second order tiles have been found by using the Frame Constraints (Theorem \ref{FC}), and their scaffolding has been recorded.

\subsubsection*{Algorithm 1.2: Completing ${\Psi}$}\label{12}
\noindent Algorithm 1.2 completes the set of tiles $T$ and the finite ordered alphabet $\Sigma$, returning a two dimensional 2-automatic sequence.  \\

\noindent {\fontfamily{cmtt}\selectfont\textbf{Step 1.2.1:} Let $\sigma_l$ be the first letter in ${\Sigma}$ such that ${\Psi}(\sigma_l)$ is not yet defined. Let $\sigma_i,\sigma_j,\sigma_k\in{\Sigma}$ be the scaffolding of $\sigma_l$. Apply ${\Psi}$ to $\sigma_i,\sigma_j$ and $\sigma_k$ to obtain the coloured parts of Figure \ref{1.2}.}  \begin{figure}[H]
    \centering
    \includegraphics[width=0.65\linewidth]{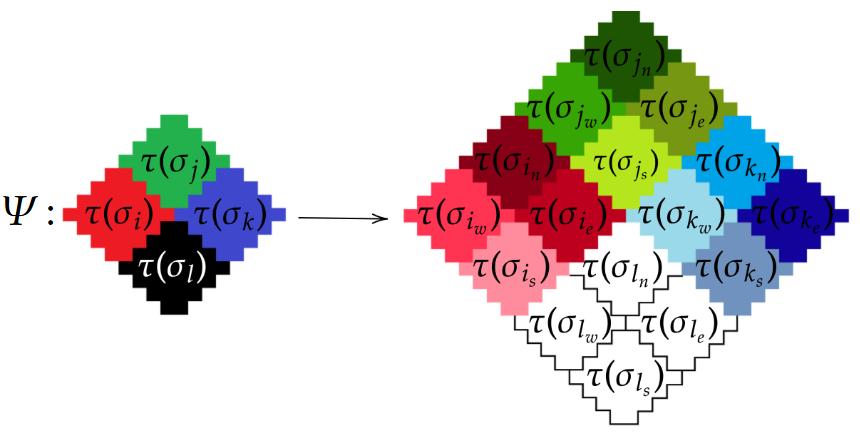}
    \caption{\label{1.2} The result of applying ${\Psi}$ to the scaffolding of $\sigma_l$.}
\end{figure}
\noindent When {\fontfamily{cmtt}\selectfont Step 1.2.1} is first run, $\Psi$ has been defined on every element of the scaffolding by Algorithm 1.1. \\

\noindent {\fontfamily{cmtt}\selectfont Step 1.2.2: Apply the Frame Constraints (Theorem \ref{FC}) to calculate the tile $t={\tau}(\sigma_{l_n})$ in \\Figure \ref{1.2}. If $t$ is not in ${T}$, append it, append $\sigma_{m(T)}$ to ${\Sigma}$ and append the scaffolding of $t$\\ to $Scaf$. Otherwise, do nothing.}\\

\noindent  There are two noteworthy comments relating to {\fontfamily{cmtt}\selectfont Step 1.2.2}. Firstly, there is no reason \textit{a priori} that the scaffolding of any $\sigma\in\Sigma$ should be unique. Algorithms 2.1 and 2.2 verify that the image of $\sigma$ under $\Psi$ does not depend on the scaffolding of $\sigma$, despite this being vital to the calculation. Secondly, if the number wall of the sequence $S$ contains windows of unbounded size, {\fontfamily{cmtt}\selectfont Step 1.2.2} will eventually fail. That is, the image of the scaffolding of $\sigma_l$ will not contain enough information to apply the Frame Constraints (Theorem \ref{FC}) when deriving the white tiles from Figure \ref{1.2}. More specifically, using notation form Figure \ref{window}, the Third Frame Constraint formula requires the values of $A_k,B_k,C_k,D_k,E_k,F_k$ and $G_k$. However, if the window is too large then some of these entries may not be contained in the image of the scaffolding of $\sigma_l$, stopping the algorithm. \\

\noindent {\fontfamily{cmtt}\selectfont\textbf{Step 1.2.3:} Define $\sigma_{l_n}:={\tau^{-1}}(t)$}. \\

\noindent{\fontfamily{cmtt}\selectfont \textbf{Step 1.2.4:} Repeat steps 1.2.2 and 1.2.3 for ${\tau}(\sigma_{l_w})$, ${\tau}(\sigma_{l_e})$ and ${\tau}(\sigma_{l_s})$ in this order.}\\

\noindent {\fontfamily{cmtt}\selectfont\textbf{Step 1.2.5:} Define ${\Psi}(\sigma_l)=\begin{matrix}\sigma_{l_w}&\sigma_{l_n}\\\sigma_{l_s}&\sigma_{l_e}\end{matrix}$.}\\

\noindent {\fontfamily{cmtt}\selectfont\textbf{Step 1.2.6:} If there remains any $\sigma\in{\Sigma}$ whose image under ${\Psi}$ is undefined, go back to Step $1.2.1$. \\Otherwise, return ${\Psi}$,${\Sigma}$, ${T}$, and ${\tau}$.}\\

\noindent In {\fontfamily{cmtt}\selectfont Step 1.2.4}, the tiles were processed in the order north, west, east and south. This ensures that $\Psi$ is always defined on the scaffolding of $\sigma_l$ by the time it is processed in {\fontfamily{cmtt}\selectfont Step 1.2.1}, as these letters were discovered earlier than $\sigma_l$. \\ 

\noindent Algorithm 1.2 terminates when $\Psi(\sigma)$ has been defined for every $\sigma\in\Sigma$. This algorithm will eventually terminate, either because there are only finitely many possible tiles with entries in $\F_q$ or because the windows have grown too large and {\fontfamily{cmtt}\selectfont Step 1.2.2} fails. Assuming the algorithm has finished without failure, what is returned is a two dimensional automatic sequence given by $\Psi^{\omega}(\sigma_1)$ and a number wall coding $\tau$. Algorithms 2.1 and 2.2 verify that $\tau(\Psi^{\omega}(\sigma_1))=W_p(S)$. 
\subsection*{Part 2: Verifying the $[2,2]$-morphism}

\noindent The $[2,2]$-morphism ${\Psi}$, the finite alphabet ${\Sigma}$, the tile set ${T}$ and the number wall coding ${\tau}$ generate an automatic sequence. The next two algorithms serve to verify that this sequence is the number wall of $S$ modulo the chosen prime $p$. Indeed, for any $k$-automatic sequence $S'$, it is expected that $W_p(S')$ is a two-dimensional $k$-automatic sequence in the sense of Figure \ref{auto_nw} (see Conjecture \ref{autonw}). If this is the case, and the size of tiles was chosen correctly, then Algorithm 1 can only return the correct $[2,2]$-morphism and finite alphabet that generates $W_p(S)$. However, without knowing if Conjecture \ref{autonw} is true, one must verify that $W_p(S)=\Psi^\omega(\sigma_1)$ manually. \\

\noindent To achieve this, Algorithm 2.1 finds every possible combination of 4 tiles that can form a 4-tuple. That is, every combination of 4 tiles that can appear in the formation depicted in the left of Figure \ref{1.2} or the right of Figure \ref{2-morph}. Note, a 4-tuple does not have to be equal to $\Psi(\sigma)$ for some $\sigma\in\Sigma$ (See Figure \ref{4-tup}). Algorithm 2.2 then verifies that each 4-tuple satisfies the Frame Constraints.

\subsubsection*{Algorithm 2.1: Finding all possible 4-tuples}\label{21}
\noindent {\fontfamily{cmtt}\selectfont\textbf{Step 2.1.1:} Define an empty ordered list $Tup$ and for every $\sigma\in{\Sigma}$ append ${\Psi}(\sigma)$. }\\

\noindent The ordered list $Tup$ will contain every $4$-tuple that appears in $\Psi^\omega(\sigma_1)$. Every such 4-tuple is contained in the image of some other 4-tuple under ${\Psi}$.\begin{figure}[H]
    \centering
    \includegraphics[scale=0.25]{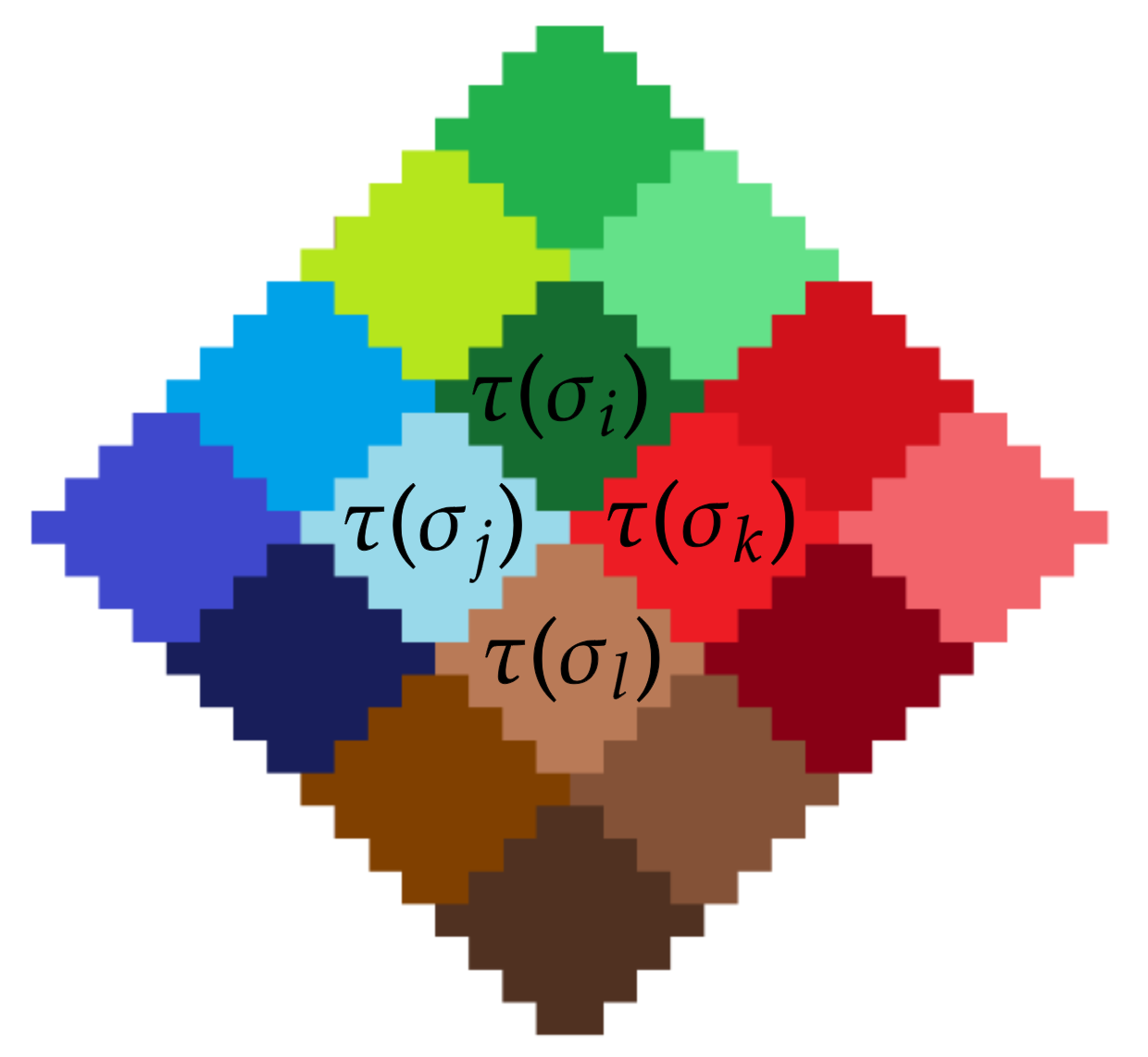}
    \caption{Each block of 4 tiles in a similar colour represent ${\tau}({\Psi}(\sigma))$ for some $\sigma\in{\Sigma}$. The 4-tuple of tiles $({\tau}(\sigma_i), {\tau}(\sigma_j), {\tau}(\sigma_k),{\tau}(\sigma_l))$ is contained within the image of these four $\sigma$ under ${\Psi}$.}
    \label{4-tup}
\end{figure}
\noindent {\fontfamily{cmtt}\selectfont\textbf{Step 2.1.2:} Let $\varsigma=\{\sigma_{i_j}: 1\le j \le 4\}$ be the first 4-tuple in $Tup$ whose image has not yet been processed. Calculate the 16-tuple ${\Psi}(\varsigma)$}.\\

\noindent {\fontfamily{cmtt}\selectfont\textbf{Step 2.1.3:}  Check every possible 4-tuple contained within ${\Psi}(\varsigma)$. If it is not in $Tup$, append it to the end.}\\

\noindent Note, there are nine such 4-tuples in total, but only five need to be checked because the four of them are images of some $\sigma\in{\Sigma}$ under ${\Psi}$ and hence are already in $Tup$.\\

\noindent {\fontfamily{cmtt}\selectfont\textbf{Step 2.1.4: }If the end of $Tup$ has been reached, return $Tup$. Otherwise, go back to step 2.1.1}\\

\noindent When the end of the list is reached, this implies that set $Tup$ is closed under $\Psi$. That is, all the 4-tuples found when applying $\Psi$ to a 4-tuple in $Tup$ are already in $Tup$. In turn, this implies $Tup$ is a complete list. 

\subsubsection*{Algorithm 2.2: Verifying the 4-tuples}\label{22}
\noindent The final step is to verify the 4-tuples satisfy the Frame Constraints (Theorem \ref{FC}). If they do, then this implies that $\tau(\Psi^{\omega}(\sigma_1))=W_p(S)$, since the two-dimensional automatic sequence of tiles satisfies the Frame Constraints (Theorem \ref{FC}) and the entries within each tile also satisfies the frame constraints by construction (Algorithm 1.2). Therefore, the two-dimensional automatic sequence $\tau(\Psi^\omega(\sigma_1))$ is the number wall of \textit{some} sequence. By construction from Algorithm 1.1, this sequence is the First-Level Paperfolding Sequence. \\

\noindent {\fontfamily{cmtt}\selectfont\textbf{Step 2.2.1:} Let $\varsigma$ be the first 4-tuple in $Tup$ that has not yet been processed by \\Algorithm 2.2. Apply ${\tau}$ to all entries of $\varsigma$ other than the southern one, denoted $\varsigma_s$.}\\

\noindent {\fontfamily{cmtt}\selectfont\textbf{Step 2.2.2:} Let $\widetilde{\tau(\varsigma_s)}$ be the tile calculated by applying the Frame Constraints (Theorem \ref{FC}) to \\the image of the scaffolding of $\varsigma_s$. If this is equal to ${\tau(\varsigma_s)}$, move to Step 2.2.3. Otherwise, \\return Failure}\\

\noindent {\fontfamily{cmtt}\selectfont\textbf{Step 2.2.3:} If $\varsigma$ was the final 4-tuple in $Tup$, return $Tup$. Otherwise, go back to step $2.2.1$.}\\

\noindent If Algorithm 2.2 is successful, then ${W_p(S)}$ is the number wall modulo $p$ of the sequence on its zeroth row. By construction, this sequence is the First Level Paperfolding sequence, and this completes the proof of automaticity. Furthermore, if the windows had unbounded size then Step 1.2.2 would have failed, hence the First Level Paperfolding sequence provides a counterexample to the $t$-$LC$ over $\F_p$. Finally, applying Theorem \ref{my_result} concludes the proof of Theorem \ref{F7}. \\

\noindent The algorithm presented in this section remains virtually unchanged when applying it to the Second Level Paperfolding Sequence. The only differences in these cases is the $2$-morphism applied in Algorithm 1.1 and the size of the tiles, which need to be 16 entries long in their largest row. The $2$-morphism that generates the Second Level Paperfolding Sequence is given by Example \ref{PF2} below. The proof is omitted. Indeed, a sceptical reader can consider the 2-morphism given in Example \ref{PF2} below to be the definition of the Second Level Paperfolding Sequence since the algorithms in Section \ref{sect:alg} are implemented without the use of equation (\ref{npf}).

\begin{example}\label{PF2}Let $\Lambda=\{A,B,C,D,E,F,G,H,I,J,K,L\}$ be an alphabet and define the 2-morphism $\chi:\Lambda\to\Lambda^*$ as \begin{align*}
   \label{PF1 sub rules}  &\chi(A)=AB& &\chi(B)=CD& &\chi(C)=EF& &\chi(D)=GD& &\chi(E)=AH& &\chi(F)=CI&\\
   &\chi(G)=EJ& &\chi(H)=CK& &\chi(I)=GI& &\chi(J)=CL& &\chi(K)=GK& &\chi(L)=GL.&
\end{align*}Let $\Delta=\{0,1\}$. Then the Second Level Paperfolding Sequence can be constructed as $\upsilon(\chi^\omega(A))$, where $\upsilon:\Lambda\to\Delta$ is the 1-coding defined by\begin{align*}
    &\upsilon(A)=0& &\upsilon(B)=0& &\upsilon(C)=1& &\upsilon(D)=0& &\upsilon(E)=2& &\upsilon(F)=1&\\
    &\upsilon(G)=3& &\upsilon(H)=2& &\upsilon(I)=1& &\upsilon(J)=3& &\upsilon(K)=2& &\upsilon(L)=3.& 
\end{align*}
\end{example}
\subsection{Results and Data}
\noindent This section details the outcome of the algorithms \hyperref[11]{1.1}, \hyperref[12]{1.2}, \hyperref[21]{2.1}, and \hyperref[22]{2.2}. For more information on how the algorithms are implemented, see Appendix \hyperref[sect:imp]{A}. To see the implementation itself, see \cite{code}. The code was run on the CSF3 super computer at the University of Manchester. See \cite{super} for the details of this instillation. \begin{table}[H]
    \centering
    \begin{tabular}{|c|c|c|c|c|}
        \hline
        Sequence&Field & Number of Tiles & Number of 4-Tuples   \\
        \hline
        $\left(\Pi^{(1)}_n\right)_{n\ge0}$&$\F_3$ & 390 & 1366  \\
        \hline
        $\left(\Pi^{(1)}_n\right)_{n\ge0}$&$\F_7$&1778011 & 17221408 \\
        \hline
        $\left(\Pi^{(1)}_n\right)_{n\ge0}$&$\F_{11}$ &70360006 & 864510531 \\
        \hline
        $\left(\Pi^{(2)}_n\right)_{n\ge0}$&$\F_5$ & 65349573& 510595180\\\hline
    \end{tabular}
\end{table}

\section{Open Problems and Conjectures} \label{sect:end}

\noindent Algorithms \hyperref[11]{1.1} and \hyperref[12]{1.2} have been run for the First (Second, respectively) Level Paperfolding sequence over $\F_{19}$ (over $\F_{13}$, respectively). Whilst they did not return errors, it did not seem like they would finish within a reasonable amount of time. However, both found over $10^9$ tiles without returning an error, giving evidence to Conjecture \ref{main_conj}. \\

\noindent The $p(t)$-$LC$ in fields of even characteristic is more of a mystery; it has been shown by brute force computation that every sequence of length 56 over $\F_2$ has a window of size greater than or equal to 3. This was too computationally exhausting to complete for larger window sizes, but it provides evidence that $t$-LC may be true over $\F_2$. One might be tempted to believe that counterexamples to $p(t)$-$LC$ exist for larger fields of even characteristic, since there is more freedom when constructing sequences than there is over $\F_3$. However, there is currently no substantial evidence for the validity of $p(t)$-$LC$ in this case. \\

\noindent Whilst developing earlier versions of the algorithms in Section 4, explicit $2$-morphisms were found that appear to generate the number wall of the Thue-Morse sequence over $\F_2$ and $\F_3$ and also the number wall of the Paperfolding sequence over $\F_2$. Similarly, explicit 3-morphisms were found that seemed to generate the number wall of the 3-automatic Cantor\footnote{Sequence A088917 on the Online Encyclopedia of Integer Sequences.} sequence over $\F_3$ and $\F_5$. However, these morphisms have not been validated as the windows in these number walls have unbounded size. The authors note that it is possible to generalise the algorithms of Section 4 to work with number walls with unbounded window size, but to do so was outside the scope of this project.  Nevertheless, there is significant empirical evidence for the following conjecture:\begin{conjecture}\label{autonw}
    The number wall of a $k$-automatic sequence over $\F_q$ is itself a two dimensional $k$-automatic sequence.
\end{conjecture}
\noindent More concrete evidence towards Conjecture \ref{autonw} was found by Allouche, Peyriere, Wen and Wen \cite{TM_auto} in 1998 when they proved via careful study of Hankel matrices that the number wall\footnote{The language of number walls is not used in this paper. However, an illustration of the number wall of the Thue Morse sequence can be found in Figure 1 of \cite{TM_auto}.} of the Thue-Morse sequence over $\F_2$ is automatic.\\ 

\noindent Proving Conjecture \ref{autonw} is a key step towards proving Conjecture \ref{main_conj}. 

\bibliographystyle{alphaurl}
\bibliography{refs}
\newpage
\section*{Appendix A: Implementation}\label{sect:imp}
\noindent The purpose of this section is to aid the reader in the understanding of the code used to prove Theorems \ref{F5} and \ref{F7}. It details how the algorithms in Section \ref{sect:alg} are implemented as code, and should be read in combination with Section \ref{sect:alg} and the codebase \cite{code}.
\subsection*{Part 0: Prerequisites}
Before discussing the details of how the implementation completes the algorithms detailed in Section \hyperref[sect:alg]{4}, one must understand the additional supporting structures that have been constructed to improve efficiency. A single tile has multiple attributes that are unique to it, all of which are required to form the notion of a `tile' in the implementations code. To bind these attributes together into a single variable, the \texttt{Tile class} has been created to allow \texttt{Tile objects} to be instantiated in the implementation, with the variables and functions outlined below. For a given $\sigma_i\in\Sigma$, a Tile object contains variables for each of the following: \begin{itemize}
    \item \texttt{id} - a unique integer allowing efficient Tile identification after the Tile has been constructed. This value is defined to be the number of Tiles that have previously been found. For example, $i+1$ is the \texttt{id} for $\sigma_i$. This is because the tiles with \texttt{id} 0 and 1 are $\tau(\sigma_{-1})$ and $\tau(\sigma_{0})$, respectively. 
\item \texttt{value} - the list of numbers that make up $\tau(\sigma_i)$, as shown in Figure \ref{zero_tile}. This variable is a 2-dimensional list of integers, where the length of the outermost list is dependent on the length of the tile (see \texttt{tile\textunderscore length} below).
\item \texttt{*\textunderscore image} - an object reference to each of the four image Tiles that make up $\Psi(\sigma_i)$, as demonstrated in Figure \ref{2-morph}. These are named \texttt{west\textunderscore image}, \texttt{north\textunderscore image}, \texttt{east\textunderscore image}, and \texttt{south\textunderscore image} in the class's code. When a Tile is first instantiated, these image Tiles have not yet been calculated and hence are denoted by $*$. Once all of the images of a Tile have been computed, the specified Tile can be considered complete for the purposes of the Tile generation process.
\item \texttt{scaffolding} - a list of Tile objects for the three tiles that sit directly above the specified Tile $T$ within the number wall, allowing computation of the \texttt{value} in $T$. This takes the place of the set $Scaf$ from Section \ref{sect:alg}.
\end{itemize}

\noindent The Tile Class also contains two static variables that are required by almost every function in the code. Hence, it is helpful to define them globally in the Tile class rather than pass them to each function. \begin{itemize}
\item \texttt{tile\textunderscore length} - this is determined by the length of the longest row in the tile. 
\item \texttt{tile\textunderscore prime} - the cardinality of the finite field that the number wall is being generated over.
\end{itemize}

\noindent The Tile class allows the implementation to vastly improve efficiency by removing the need for data to be duplicated across multiple different data structures. Without it, the implementation would need to split the various data within Tile across individual data structures, which would require additional computation to search through and synchronise.

\subsection*{Part 1: Computing the Parameters}
The remainder of the implementation approximately matches the steps described in \hyperref[11]{Part 1} and \hyperref[21]{Part 2} of the algorithm description in Section \hyperref[sect:alg]{4}. The implementation has been written to be sequence and tile length agnostic, but for the purpose of this paper it should be considered to be computed using the First Level Paperfolding Sequence with a tile length of 8.
\subsubsection*{Algorithm 1.1: Initial Conditions}
\noindent Before describing Part 1 of the implementation, several more critical data structures are defined: \begin{itemize}
\item \texttt{tiles} - a hash map (dictionary data structure in the Python implementation) of every unique Tile object computed so far, including Tiles that are yet to have their image Tiles computed. Hash maps store data using a \textit{key:value} system, where providing a given \textit{key} retrieves the related \textit{value} with an average time complexity\footnote{The worst-case scenario for accessing data in a hash map occurs when the hashing algorithm maps every key to the same location, which gives $O(n)$ time complexity for accessing data, but would require an extremely inefficient hashing algorithm for this to occur.} of $O(1)$. This allows the implementation to work efficiently when checking whether a newly generated Tile value is unique or not - as the \texttt{tiles} dictionary uses the Tile value for the \textit{key}, and a reference to the Tile itself as the \textit{value}. This takes the place of the ordered list $T$ in Section \ref{sect:alg}.
\item \texttt{new\textunderscore tiles} - a list containing every Tile object that has yet to have its image Tiles computed (henceforth referred to as \textit{Incomplete Tiles}). This forms a backlog of Tiles that require computation, which increases in length whenever a new unique Tile is computed.
\item \texttt{tiles\textunderscore by\textunderscore index} - an ordered list storing Tile objects by ascending \texttt{id} values. This list is used exclusively during Step 1.1 to aid the computation of Tiles along the zeroth row.
\end{itemize}

\noindent After declaring these data structures, the \texttt{input\textunderscore generator} function instantiates the \texttt{Zeroth Tile} and the \texttt{First Tile} as shown in Figure \ref{zero_tile}. For simplicity, these Tiles are hard-coded in the function, as they will always be required no matter the sequence or tile length.\\

The function then executes the instructions detailed in {\fontfamily{cmtt}\selectfont\textbf{Step 1.1.2}} and {\fontfamily{cmtt}\selectfont\textbf{Step 1.1.3}}. If any of the image Tiles generated in these steps are new Tiles then they are inserted into \texttt{tiles} and \texttt{new\textunderscore tiles}.

\subsubsection*{Algorithm 1.2: Generating the Complete Set of Tiles}\label{gen_tiles}
The core tiling generation is handled by the \texttt{tile\textunderscore computation} function, which takes the \texttt{tiles} and \texttt{new\textunderscore tiles} data structures as input. At this point, these data structures contain the Tiles generated using the sequence and the image of these Tiles under the $[2,2]$-morphism $\Psi$. Combined with the Frame Constraints (Theorem \ref{FC}), this is enough data to compute every other unique Tile in the specified number wall following the steps outlined in Algorithm \hyperref[12]{1.2}.\\

\noindent The \texttt{tile\textunderscore computation} function works by processing the incomplete Tiles within \texttt{new\textunderscore tiles} in \textit{First-In-First-Out} (FIFO) order. By following FIFO ordering, it is guaranteed that for each Tile in the scaffolding of the incomplete Tile, all of the required image Tiles will have already been computed and recorded - without this, the program would break as soon as an incomplete scaffold Tile was used.\\

\noindent For each incomplete Tile $T$ within \texttt{new\textunderscore tiles}, the three scaffolding Tiles of $T$ are collated. Then for each scaffolding Tile $S$, the four image Tiles of $S$ are stitched together in their related positions (in correlation to the number wall), where the resultant nested list forms a grid of number wall values the shape of the coloured area in the upper-right section of Figure \ref{tile_computation}. This gives a large enough portion of the number wall to apply the Frame Constraints to generate the \texttt{value} variables for $T$'s image Tiles, allowing those Tiles to be instantiated. This process is referred to as the \textbf{Scaffolding-Tile Generation} technique.

\begin{figure}[H]
    \centering
    \includegraphics[width=0.5\linewidth]{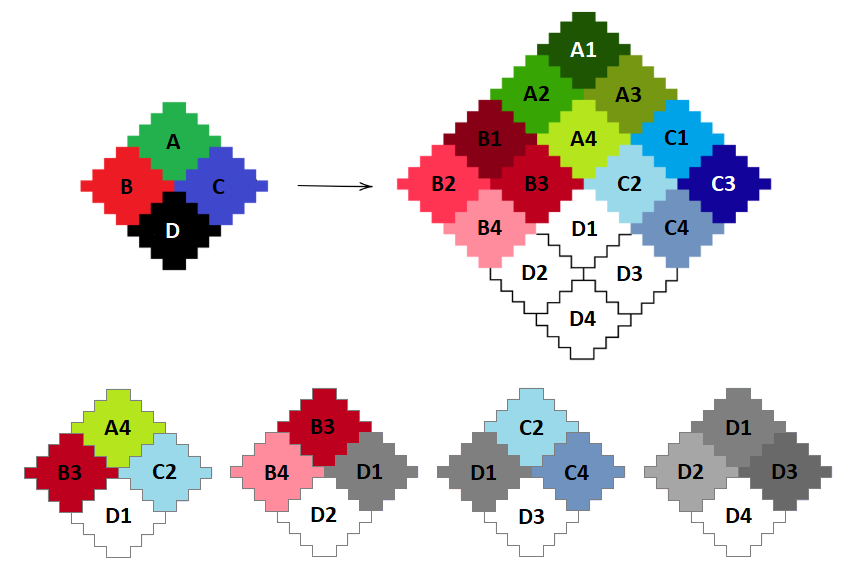}
    \caption{\label{tile_computation} {Upper -} The results of applying the Scaffolding-Tile Generation technique on incomplete Tile D, with scaffolding Tiles A, B, and C. {Lower -} The Tiles referenced within the \texttt{scaffolding} variable of each resultant image Tile from D.}
\end{figure}

\noindent Once the image Tiles of incomplete Tile $T$ have had their \texttt{value} variables computed, the function performs the following steps for each image:\begin{enumerate}
  \item Check whether the tile \texttt{value} is unseen in the execution of the implementation so far. This is completed efficiently by leveraging the search capabilities of the \texttt{tiles} dictionary: \begin{itemize}
    \item To search a dictionary, a \textit{key} is provided, which the code will then complete a \textit{hashing function} on. The result of the hashing function on the key provides a memory address that is unique to that specific key, allowing the \textit{value} associated with that key to be retrieved.
    \item The \texttt{input\textunderscore generator} function utilises this process to check Tile uniqueness by using the image Tile's computed \texttt{value} as the key to \texttt{tiles}. If data is returned from this computation then the Tile has already been found in an earlier computation, whereas if a null value is returned then this memory address has not been written to before, implying this must be a new unique Tile.
  \end{itemize}
  \item If the image \texttt{value} is unseen, instantiate the image into a Tile object, assigning the related Tiles to \texttt{scaffolding} as shown in the lower section of Figure \ref{tile_computation}. Then append the Tile into \texttt{tiles} and \texttt{new\textunderscore tiles}.
\end{enumerate}

\noindent After the {Scaffolding-Tile Generation} has executed, the incomplete Tile $T$ has each of its image Tiles assigned to it, and can now be considered complete\footnote{To maximise memory efficiency, Tile $T$ is removed from the \texttt{new\textunderscore tiles} list upon its retrieval, before any processing of the Tile has been performed.}. In cases where one of the image Tiles has been seen previously, a reference to the Tile object is retrieved from \texttt{tiles} for the purpose of assigning the image to $T$.\\

\noindent The \texttt{input\textunderscore generator} function will have finished processing the image Tiles of every unique Tile once the \texttt{new\textunderscore tiles} list is empty, at which point every unique Tile specific to this sequence, Tile length, and finite field has been found, along with its image under $\Psi$. The function then returns the \texttt{tiles} dictionary.

\subsection*{Part 2: Verifying the $[2,2]$-morphism}
Once Part 1 has finished computation, the \texttt{tiles} dictionary contains a complete set of unique Tiles and their images under $\Psi$. However, it must now be verified that the two dimensional automatic sequence provided by the successful execution of the \texttt{tile\textunderscore computation} is equal to the number wall. 
\subsubsection*{Algorithm 2.1: Generating all possible 4-tuples}
Before identifying any currently unrepresented 4-tuples in the data, the \texttt{generate\textunderscore four\textunderscore tuples} function needs to convert the information in the \texttt{tiles} dictionary into a different representation focused around 4-tuples\footnote{Note that in the implementation 4-tuple does not use a custom Class, it is simply a \textit{tuple} containing the \texttt{id} of the four Tiles that it represents.}. To achieve this, two new data structures are defined:
\begin{itemize}
    \item \texttt{tuples\textunderscore by\textunderscore index} - a list of each currently known 4-tuple, where the existing 4-tuples are generated from the four image Tiles of each unique Tile in the \texttt{tiles} dictionary.
    \item \texttt{tuples} - a dictionary of every unique 4-tuple computed so far, where both the \textit{key} and \textit{value} of the entries are the 4-tuples themselves. This dictionary is immediately populated with the existing known 4-tuples from \texttt{tuples\textunderscore by\textunderscore index}. The sole purpose of this data structure is to allow the uniqueness of each newly generated 4-tuple to be verified efficiently.
\end{itemize}

\noindent Once these data structures have been instantiated, the \texttt{generate\textunderscore four\textunderscore tuples} function iterates through every 4-tuple in \texttt{tuples\textunderscore by\textunderscore index} and converts all of the images of each Tile in the 4-tuple into a singular grid. This essentially forms a 16-tuple (or \textit{image tuple}), and mirrors the diagram shown in the upper row of Figure \ref{tile_computation}. However, for simplicity in the implementation, the diamond shape of the image tuple is rotated onto its side, allowing a square nested list to be used to represent the 2-dimensional space. Similarly, to save computation, this function does not handle the raw Tile \texttt{values}. Instead each tile is represented by its Tile \texttt{id} during each step of the computation.\\

\noindent From the image tuple, 5 new, potentially unseen 4-tuples are sampled. These are generated from the confluence points between Tile images from the original 4-tuple, i.e. where unrelated image 4-tuples meet. Using the top diagram in Figure \ref{tile_computation} as an example, the following new 4-tuples would be sampled:
\begin{itemize}
    \item \textbf{Upper tuple} - containing the Tiles B1, B3, A2, A4.
    \item \textbf{Right tuple} - containing the Tiles A3, A4, C1, C2.
    \item \textbf{Left tuple} - containing the Tiles B3, B4, D1, D2.
    \item \textbf{Lower tuple} - containing the Tiles D1, D3, C2, C4.
    \item \textbf{Middle tuple} - containing the Tiles A4, B3, C2, D1.
\end{itemize}

\noindent Each of these new 4-tuples is then checked for uniqueness, using the same hash map technique described in \hyperref[gen_tiles]{Algorithm 1.2}. If any of the new 4-tuples are not currently contained in \texttt{tuples} (and are thus previously unseen), they are added to both \texttt{tuples} and \texttt{tuples\textunderscore by\textunderscore index}. Once every 4-tuple in \texttt{tuples\textunderscore by\textunderscore index} has been processed, all the unique 4-tuples that were previously unrepresented in the \texttt{tiles} data structure will have been identified. The \texttt{generate\textunderscore four\textunderscore tuples} function then returns the \texttt{tuples\textunderscore by\textunderscore index} list.

\subsubsection*{Algorithm 2.2: Verifying the 4-tuples}
The final stage of the implementation is to verify that all of the identified 4-tuples conform to the \hyperref[FC]{Frame Constraints}. The \texttt{verify\textunderscore four\textunderscore tuples} function takes the \texttt{tuples\textunderscore by\textunderscore index} list and \texttt{tiles} dictionary as inputs, and instantiates the following data structure to aid in the verification process:
\begin{itemize}
    \item \texttt{tiles\textunderscore by\textunderscore index} - a pre-populated list of each Tile identified in Algorithm 1.2 (copied from \texttt{tiles}), where each item in the list is a reference to a Tile object, ordered by ascending Tile \texttt{id}. The list is used to assist in populating 4-tuples with Tile \texttt{value}s, using the Tile \texttt{id}s already contained in each 4-tuple to retrieve the related Tiles.
\end{itemize}

\noindent The core of the function is completed by iterating through each 4-tuple in \texttt{tuples\textunderscore by\textunderscore index}, skipping the \texttt{Zeroth Tile} and \texttt{First Tile} (which cannot be generated using the Frame Constraints). For each 4-tuple that is processed, the function generates a square 2-dimensional list of junk data with a length twice as large as \texttt{tile\textunderscore length}. This is generated to serve as a blank slate for inserting Tile \texttt{value}s into, and is square to reduce the complexity of the list indexing required to construct each Tile's \texttt{value} in the correct location. Following this, the \texttt{value} for the upper, left, and right Tile from the 4-tuple are inserted into the square - at which point the 4-tuple is ready to be verified. \\

\noindent To assess the correctness of the prepared 4-tuple, the \texttt{verify\textunderscore four\textunderscore tuples} function utilises the same Scaffolding-Tile Generation technique from {Algorithm 1.2} - using the Frame Constraints with the three inserted Tiles to generate the missing, lower Tile. Using Figure \ref{tile_computation} as an example, this would involve inputting Tiles A, B, and C as a scaffolding to generate Tile D in the top left portion of the diagram. This generated fourth Tile is then compared to the actual \texttt{value} of the fourth tile in the 4-tuple, which can be found using the Tile \texttt{id} stored in the lower portion of the 4-tuple. If the two Tiles do not have matching \texttt{value}s, then the verification of that 4-tuple is considered to have failed. \\

\noindent If any of the 4-tuples fail their verification, then the entire function exits with a failure result. However, if every single 4-tuple is processed successfully by the verification function, then the entire set of Tiles and substitution rules should be considered accurate for the sequence, finite field, and tile length.

~\\
~\\
~\\

\raisebox{-10ex}
{
    \begin{minipage}{7cm}
        \textbf{Samuel Garrett}\\
        samuel.garrett96@gmail.com
    \end{minipage}
} 
\raisebox{-10ex}
{
    \begin{minipage}{10cm}
        \textbf{Steven Robertson}\\
        School of Mathematics\\
        The University of Manchester\\
        Alan Turing Building\\
        Manchester, M13 9PL\\
        United Kingdom\\
        \texttt{steven.robertson@manchester.ac.uk}  
    \end{minipage}
} 
\end{document}